\documentclass[11pt]{amsart}
\usepackage{amscd,amssymb,longtable, rotating, lscape, graphicx}

\usepackage{supertabular}
\usepackage{setspace}

\title{Unstable singular del Pezzo hypersurfaces with lower index}
\author[I.-K. Kim]{In-Kyun Kim}
\address{I.-K. Kim : Mathematics, Sungkyunkwan University, 2066 Seobu-ro, Suwon 16419, Korea}
\email{soulcraw@gmail.com}

\author[J. Won]{Joonyeong Won}
\address{J. Won : Center for Mathematical Challenges, Korea Institute for Advanced Study, Seoul 02455, Republic of Korea}
\email{leonwon@kias.re.kr}

\usepackage{amsthm,amsmath}
\usepackage{amscd,amssymb,graphicx}
\usepackage{pdfsync}
\usepackage[all]{xy}
\usepackage{tikz-cd}
\usepackage{color}
\usepackage{mathtools}
\usepackage{microtype}
\usepackage{IEEEtrantools}

\newtheorem{theorem}{Theorem}[section]
\newtheorem{definition}[theorem]{Definition}
\newtheorem{corollary}[theorem]{Corollary}

\newtheorem{lemma}[theorem]{Lemma}
\newtheorem{conj}[theorem]{Conjecture}

\theoremstyle{remark}

\newcommand{\lct}{\operatorname{lct}}

\newcommand{\wt}{\operatorname{wt}}
\newcommand{\vol}{\operatorname{vol}}
\newcommand{\dd}{\mathop{}\!\mathrm{d}}

\newcommand{\CC}{{\mathbb C}}

\newcommand{\NN}{{\mathbb N}}
\newcommand{\PP}{{\mathbb P}}
\newcommand{\QQ}{{\mathbb Q}}
\newcommand{\RR}{{\mathbb R}}
\newcommand{\ZZ}{{\mathbb Z}}

\newcommand{\msp}{\mathsf p}

\newcommand{\mcO}{\mathcal O}

\begin{document}
\maketitle
\begin{abstract}
    We give examples of K-unstable singular del Pezzo surfaces which are weighted hypersurfaces with index 2.
\end{abstract}

Throughout the article, the ground field is assumed to be the field of complex numbers.

\section{Introduction}
Singular del Pezzo surfaces with orbifold K\"ahler-Einstein metrics attracted the attention from Sasakian geometry since a link of singularity attach it to real 5-manifolds with Sasaki-Einstein metrics (\cite{BG08}, \cite{BGN03}, \cite{PW19} ).

Let $S_d$ be a quasi-smooth and well-formed hypersurface in $\mathbb{P}(a_0, a_1, a_2, a_3 )$ of degree $d$, where $a_0\leq a_1 \leq a_2 \leq a_3 $ are positive integers. Put an index $I=a_0+a_1+a_2+a_3-d$ and assume $I$ is positive. Then $S_d$ is a log del Pezzo surface with at most quotient singularities. For $I=1$ Johnson and Koll\'ar \cite{JK01} found all possibilities for quintuple $(a_0,a_1,a_2,a_3,d)$ and then computed the alpha invariant to show the existence of the orbifold K\"ahler-Einstein metric in the case when the quintuple $(a_0,a_1,a_2,a_3,d)$ is not of the following four quintuples : $ (1, 2, 3, 5, 10),(1, 3, 5, 7, 15),(1, 3, 5, 8, 16),(2, 3, 5, 9, 18)$. Later, Araujo \cite{A02} shows for the two of these four cases. 
The remaining two cases have been dealt with in the paper \cite{CPS10} who  shows the alpha invariant is large enough to admit orbifold K\"ahler-Einstein metrics except for possibly the case when $(a_0,a_1,a_2,a_3,d)=(1,3,5,7,15) $ and the defining equation of the surface $S_d$ does not contain a monomial $yzt$. Finally, the paper \cite{CPS18} proves K-stability of remaining one case  by estimating the delta invariant so that it has orbifold K\"ahler-Einstein metrics. This ends the existence of the orbifold K\"ahler-Einstein metric on the log del Pezzo hypersurfaces with index 1. 

Moreover the paper \cite{CPS18} expect that the similar phenomenon happens in  surfaces with lower index.

\begin{conj}\label{conj1} \cite{CPS18}  
If $I=2$ then $S_d$ admits an orbifold K\"ahler-Einstein metric.
 \end{conj}
 The present article is motivated by Conjecture \ref{conj1} and answer it by following theorem. 
 
\begin{theorem}
    Suppose that $S_d$ is quasi-smooth and the quintuple $(a_0,a_1,a_2,a_3,d)$ is one of the following quintuples:
    \begin{equation*}
        (1,6,9,13,27),\quad (1,9,15,22), \quad (1,3,3n+3,3n+4,6,+9),
    \end{equation*}
    \begin{equation*}
        (1,1,n+1,m+1,n+m+2),\quad(1,3,3n+4,3n+5,6n+11)
    \end{equation*}
    where $n$ and $m$ are non-negative integers with $n<m$. Then $S_d$ does not have an orbifold K\"ahler-Einstein metric.
\end{theorem}

 \begin{conj}
A singular del Pezzo surface $S_d$ with $I=2$ and the quintuple $(a_0,a_1,a_2,a_3,d)$ is not one of the above 5 cases and is not $(1,1,n+1, n+1, 2n+2)$. Then it is $K$-stable so that it admits an orbifold K\"ahler-Einstein metric. 
 \end{conj}

For the case $(1,1,n+1, n+1, 2n+2)$, the delta invariant of the corresponding surface less than or equal to one as it is verified in   Section 4. But  we expect the following; 

 \begin{conj}
Suppose that $S_n$ is quasi-smooth and the quintuple $(a_0,a_1,a_2,a_3,d)$ is $(1,1,n+1, n+1, 2n+2)$.  Then we have that 
$\delta(S_n)=1$ so that it is K-semistable. Moreover the surface   have an orbifold K\"ahler-Einstein metric for any $n$.
 \end{conj}
 
 {\em{Acknowledgement.}}
 The first author was supported by the National Research Foundation of Korea(NRF) grant funded by the Korea government(MSIP) (NRF-2020R1A2C4002510). The second author was supported by the National Research Foundation of
Korea(NRF-2020R1A2C1A01008018) and a KIAS Individual Grant (SP037003) via the Center
for Mathematical Challenges at Korea Institute for Advanced Study.

\section{Preliminary}
\subsection{Notation}
Throughout the paper we use the following notations:
\begin{itemize}
    \item For positive integers $a_0$, $a_1$, $a_2$ and $a_3$, $\PP(a_0,a_1,a_2,a_3)$ is the weighted projective space. We assume that $a_0\leq a_1\leq a_2\leq a_3$.
    \item We usually write $x$, $y$, $z$ and $t$ for the weighted homogeneous coordinates of $\PP(a_0,a_1,a_2,a_3)$ with weights $\wt(x)=a_0$, $\wt(y)=a_1$, $\wt(z)=a_2$ and $\wt(t)=a_3$.
    \item $S_d \subset \PP(a_0,a_1,a_2,a_3)$ denotes a quasi-smooth weighted hypersurface given by a quasi-homogenious polynomial of degree $d$.
    \item $H_{*}$ is the hyperplane that is cut out by the equation $* = 0$ in $S_d$.
    \item $\msp_x$ denotes the point on $S_d$ given by $y=z=t=0$. The points $\msp_y$, $\msp_z$ and $\msp_t$ are defined in a similar way.
    \item $-K_{S_d}$ denotes the anti-canonical divisor of $S_d$.
\end{itemize}

\subsection{Foundation}
Let $X$ be a $\QQ$-Fano variety, i.e.\ a normal projective $\QQ$-factorial variety with at most terminal singularities such that $-K_X$ is ample.
By \cite[Theorem A]{BJ20}, we have the following:
\begin{equation}\label{alpha and delta inequalities}
    \frac{\dim(X) + 1}{\dim(X)}\alpha(X)\leq \delta(X) \leq (\dim(X) + 1)\alpha(X).
\end{equation}

\begin{theorem}[{\cite[Theorem B]{BJ20}}]\label{delta and k-stable}
    Let $X$ be a $\QQ$-Fano variety.
    \begin{itemize}
        \item $X$ is $K$-semistable if and only if $\delta(X)\geq 1$;
        \item $X$ is uniformly $K$-stable if and only if $\delta(X) > 1$.
    \end{itemize}
\end{theorem}

\begin{definition}
    We define the pseudo-effective threshold $\tau(E)\in \RR_{>0}$ of $E$ with respect to $-K_X$ as
    \begin{equation*}
        \tau(E) = \sup\{x\in\RR_{>0} ~|~ \vol_{X}(-K_X - xE) > 0\}.
    \end{equation*}
\end{definition}

\begin{theorem}[\cite{F19b}]\label{beta invariant}
    Let $X$ be a $n$-dimensional $\QQ$-Fano variety. For an arbitrary prime divisor $E$ over $X$, we set
    \begin{equation*}
        \beta(E) = A_X(E)(-K_X)^n - \int_{0}^{\tau(E)}\vol_X(\pi^*(-K_X) - x E) \dd x
    \end{equation*}
    where $A_X(E)$ is the log discrepancy of $E$.
    Then we have $(X, -K_X)$ is $K$-semistable if and only if $\beta(E)\geq 0$ for any $E$.
\end{theorem}

\begin{definition}
Let $(X,D)$ be a pair, that is, $D$ is an effective $\QQ$-divisor, and let $\msp \in X$ be a point. We define the {\it log canonical threshold} (LCT, for short) of $(X,D)$ and the {\it log canonical threshold} of $(X,D)$ {\it at} $\msp$ to be the numbers
\[
\begin{split}
\lct (X,D) &= \sup \{\, c \mid \text{$(X, c D)$ is log canonical} \,\}, \\
\lct_{\msp} (X,D) &= \sup \{\, c \mid \text{$(X, c D)$ is log canonical at $\msp$} \,\},
\end{split}
\]
respectively.
We define
\[
\lct_{\msp} (X) = \inf \{\, \lct_{\msp} (X,D) \mid D \text{ is an effective $\QQ$-divisor}, D \equiv -K_X \,\},
\]
and for a subset $\Sigma \subset X$, we define
\[
\lct_{\Sigma} (X) = \inf \{\, \lct_{\msp} (X) \mid \msp \in \Sigma \, \}.
\]
The number $\alpha (X) := \lct_X (X)$ is called the {\it global log canonical threshold} (GLCT, for short) or the {\it alpha invariant} of $X$.
\end{definition} 

Let $X$ be a surface with quotient singularities. And let $D$ be an effective $\QQ$-divisor. If $D$ is Cartier then its volume is the number
\begin{equation}\label{volume}
    \vol(D) = \limsup_{k\in \NN}\frac{h^0(\mcO_X(kD))}{k^2/ 2!}.
\end{equation}
When $D$ is a $\QQ$-divisor one can define the volume $\vol(D)$ of $D$ as in (\ref{volume}) taking the $\limsup$ over those $k$ for which $kD$ is integral. Moreover we can define its volume using the identity
\begin{equation*}
    \vol(D) = \frac{1}{\lambda^2}\vol(\lambda D)
\end{equation*}
for an appropriate positive rational number $\lambda$. The volume of $D$ depends only upon its numerical equivalence class (see \cite{Laz-positivity-in-AG} for details).

If $D$ is not pseudoeffective, then $\vol(D) = 0$. If $D$ is pseudoeffective, its volume can be computed using its Zariski decomposition \cite{Prokhorov, BauerKuronyaSzemberg}. Namely, if $D$ is pseudoeffective, then there exists a nef $\RR$-divisor $P$ on the surface $X$ such that
\begin{equation*}
    D\sim_{\RR} P + \sum_{i=1}^n a_i C_i
\end{equation*}
where each $C_i$ is an irreducible curve on $X$ with $P\cdot C_i$, each $a_i$ is a non-negative real number, and the intersection form of the curves $C_1,\ldots ,C_n$ is negative definite. Such decomposition is unique, and it follows from [BKS04, Corollary 3.2] that
\begin{equation*}
    \vol(D) = \vol(P) = P^2.
\end{equation*}

We consider a cyclic quotient singularity $V = \CC^2 /\ZZ_m(a_1, a_2)$, where $a_1$ and $a_2$ are positive integers which are coprime. Let $x_1$ and $x_2$ be weighted coordinates of $V$ with weights $\wt(x_1) = a_1$ and $\wt(x_2) = a_2$. And let $\phi\colon W\to V$ be the weighted blow-up at the origin of $V$ with weights $\wt(x_1) = a_1$ and $\wt(x_2) = a_2$. Then we have the following:
\begin{equation*}
    K_W\equiv \phi^*(K_V) + \left(-1 + \frac{a_1}{m} + \frac{a_2}{m}\right)E
\end{equation*}
where $E$ is the exceptional divisor of $\phi$ and
\begin{equation*}
    E^2 = -\frac{m}{a_1a_2}.
\end{equation*}
Let $H_{x_i}$ be the hyperplane that is cut out by $x_i = 0$ in $V$. Then we have
\begin{equation*}
    \phi^*(H_{x_i}) = \bar{H}_{x_i} + \frac{a_i}{m}E
\end{equation*}
where $\bar{H}_{x_i}$ is the strict transform of $H_{x_i}$.

\section{Singular del Pezzo surfaces with small $\alpha$-invariant}
In this section, we estimate the $\alpha$-invariants of the following quasi-smooth weighted hypersurfaces:
\begin{itemize}
    \item $S_{27}$ : a quasi-smooth weighted hypersurface in $\PP(1,6,9,13)$ of degree $27$;
    \item $S_{45}$ : a quasi-smooth weighted hypersurface in $\PP(1,9,15,22)$ of degree $45$;
    \item $S_{6n+9}$ : a quasi-smooth weighted hypersurface in $\PP(1,3,3n+3,3n+4)$ of degree $6n+9$
\end{itemize}
where $n$ is a positive integer. By suitable coordinate changes we can assume that $S_{27}$ is given by a quasi-homogeneous polynomial
\begin{equation*}
    t^2x + z^3 + z^2f_9(x,y) + zf_{18}(x,y) + f_{27}(x,y) = 0
\end{equation*}
where each $f_i(x,y)$ is a quasi-homogeneous polynomial of degree $i$ with weights $\wt(x) = 1$, $\wt(y) = 6$, $\wt(z) = 9$ and $\wt(t) = 13$, $S_{45}$ by
\begin{equation*}
    z^3 + y^5 + xf(x,y,z,t) = 0
\end{equation*}
where $f(x,y,z,t)$ is a quasi-homogeneous polynomial of degree $44$ with weights $\wt(x) = 1$, $\wt(y) = 9$, $\wt(z) = 15$ and $\wt(t) = 22$ and $S_{6n+9}$ by
\begin{equation*}
    t^2x + txf(x,y,z) + z^2y + azy^{n+2} + by^{2n+3} + xg(x,y,z) = 0
\end{equation*}
where $a$ or $b$ are non-zero constants with weights $\wt(x) = 1$, $\wt(y) = 3$, $\wt(z) = 3n+3$ and $\wt(t) = 3n+4$.

For the convenience, we use $S$ for the surfaces $S_{27}$, $S_{45}$ and $S_{6n+9}$. Let $H_x$ be the hyperplane that is cut out by $x = 0$ in $S$. Then it is isomorphic to the variety
\begin{equation*}
    (z^3 + zy^3 = 0)\subset \PP(6,9,13) \textrm{~for~} S_{27},
\end{equation*}
\begin{equation*}
    (z^3 + y^5 = 0)\subset \PP(9,15,22) \textrm{~for~} S_{45},
\end{equation*}
\begin{equation*}
    (z^2y + azy^{n+2} + by^{2n+3} = 0)\subset \PP(3,3n+3,3n+4) \textrm{~for~} S_{6n+9}.
\end{equation*}
Then it is easy the see that 
\begin{equation*}
    \lct(S_{27}, H_x) = \frac{5}{9},
\end{equation*}
\begin{equation*}
    \lct(S_{45}, H_x) = \frac{8}{15},
\end{equation*}
\begin{equation*}
    \lct(S_{6n+9}, H_x) = \frac{n+2}{2n+3}
\end{equation*}
which implies that $\alpha(S)<\frac{1}{3}$.
\begin{corollary}
    The singular del Pezzo surfaces $S_{27}$, $S_{45}$ and $S_{6n+9}$ are not $K$-semistable. Furthermore, they do not have an orbifold K\"alher-Einstein metric.
\end{corollary}
\begin{proof}
    By the inequality (\ref{alpha and delta inequalities}) and Lemma \ref{delta and k-stable} we see that the singular del Pezzo surfaces $S_{27}$, $S_{45}$ and $S_{6n+9}$ are not $K$-semistable.
\end{proof}

\section{Singular del Pezzo surfaces with negative $\beta$-invariant}
In this section, we prove that there are prime divisors $E$ over the following quasi-smooth weighted hypersurfaces such that $\beta(E) < 0$.

The following quasi-smooth weighted hypersurface is the case of $n = 0$ for the quintuple $(1,3,3n+3,3n+4,6n+9)$. 
\begin{lemma}
    Let $S\subset\PP(1, 3, 3, 4)$ be a quasi-smooth weighted hypersurface of degree $9$. Then $S$ is not $K$-semistable.
\end{lemma}

\begin{proof}
    By a suitable coordinate change we can assume that $S$ is given by a quasi-homogeneous polynomial
    \begin{equation*}
        t^2x + yz(ay + bz) + x^3f(x,y,z) = 0
    \end{equation*}
    where $a$ or $b$ are non-zero constants with weights $\wt(x) = 1$, $\wt(y) = 3$, $\wt(z) = 3$ and $\wt(t) = 4$ and $f(x,y,z)$ is a quasi-homogeneous polynomial of degree $6$. Let $H_x$ be the hyperplane that is cut out by $x = 0$ in $S$. We write
    \[
        H_x = L_1 + L_2 + L_3
    \]
    where $L_1$, $L_2$ and $L_3$ are the curves that are given by $x=y=0$, $x=z=0$ and $x=ay+bz=0$ in $\PP(1,3,3,4)$.
    
    Meanwhile. $S$ is singular at the point $\msp_t$ of type $\frac{1}{4}(1,1)$.

    In a neighborhood of $\msp_t$, we may regard that $y$ and $z$ are local weighted coordinates. Let $\pi\colon Y\to S$ be the weighted blow-up at $\msp_t$ with weights $\wt(y) = 1$ and $\wt(z) = 1$. Then we have
    \[
        K_Y \equiv \pi^*(K_S) - \frac{1}{2}E
    \]
    where $E$ is the exceptional divisor of $\pi$. And we have
    \begin{equation*}
        \pi^*(H_x) = \bar{H}_x = \bar{L}_1 + \bar{L}_2 + \bar{L}_3
    \end{equation*}
    where $\bar{H}_x$, $\bar{L}_1$, $\bar{L}_2$ and $\bar{L}_3$ are strict transforms of $H_x$, $L_1$, $L_2$ and $L_3$, respectively. Then we have the following intersection numbers:
    \begin{equation*}
        E^2 = -4,\qquad \bar{L}_1^2 = \bar{L}_2^2 = \bar{L}_3^2 = -\frac{2}{3},\qquad \bar{L}_1\cdot \bar{L}_2 = \bar{L}_2\cdot\bar{L}_3 = \bar{L}_3\cdot \bar{L}_1 = 0.
    \end{equation*}
    We claim that 
    \begin{equation*}
        \beta(E) = A_S(E)(-K_S)^2 - \int_{0}^{\tau (E)}\vol(\pi^*(-K_S) - \lambda E) \dd \lambda
    \end{equation*}
    is not positive. To prove this, we consider the $\QQ$-divisor
    \begin{equation*}
        \pi^*(-K_S) - \lambda E\equiv 2\bar{L}_1 + 2\bar{L}_2 + 2\bar{L}_3 + \left(\frac{3}{2} - \lambda\right)E.
    \end{equation*}
    Since $\bar{L}_i^2 < 0$ for each $\bar{L}_i$, we have $\tau(E) = \frac{3}{2}$, that is, $\vol(\pi^*(-K_S) - \lambda E) = 0$ for $\lambda < \tau(E)$. 

    We have
    \begin{equation*}
        \left(2\bar{L}_1 + 2\bar{L}_2 + 2\bar{L}_3 + \left(\frac{3}{2} - \lambda \right)E\right)\cdot \bar{L}_i = \frac{1}{6} - \lambda
    \end{equation*}
    for $i=1,2,3$. They imply that $\pi^*(-K_S) - \lambda E$ is nef for $\lambda\leq \frac{1}{6}$. Thus we have
    \begin{equation*}
        \vol(\pi^*(-K_S) - \lambda E) = 1 - 4\lambda^2
    \end{equation*}
    for $\lambda\leq \frac{1}{6}$. 

    We assume that $\frac{1}{6}\leq \lambda \leq \frac{3}{2}$. Since $\left(\frac{3}{2}\left(\bar{L}_1 + \bar{L}_2 + \bar{L}_3\right) + E\right)\cdot \bar{L}_i = 0$ for $i=1,2,3$, and the intersection form of $\bar{L}_1$, $\bar{L}_2$, $\bar{L}_3$ is negative definite, the positive term of the Zariski decomposition of $\pi^*(-K_S) - \lambda E$ is
    \begin{equation*}
        \left(\frac{3}{2} - \lambda\right)\left(\frac{3}{2}\left(\bar{L}_1 + \bar{L}_2 + \bar{L}_3\right) + E\right).
    \end{equation*}
    Thus
    \begin{equation*}
        \vol(\pi^*(-K_S) - \lambda E) = \frac{1}{2}\left(\frac{3}{2} - \lambda \right)^2
    \end{equation*}
    for $\frac{1}{6}\leq \lambda \leq \frac{3}{2}$. From the above equations we have

    \begin{equation*}
        \displaystyle\int_{0}^{\tau(E)}\vol(\pi^*(-K_S) - \lambda E) \dd \lambda =\displaystyle\int_{0}^{\frac{1}{6}} 1 - 4\lambda^2 \dd \lambda + \displaystyle\int_{\frac{1}{6}}^{\frac{3}{2}} \frac{1}{2}\left(\frac{3}{2} - \lambda\right)^2 \dd \lambda = \frac{5}{9}.
    \end{equation*}
    Moreover, we have $A_S(E)(-K_S)^2 = \frac{1}{2}$. As a result, we see that
    \begin{equation*}
        \beta(E) = -\frac{1}{18} < 0.
    \end{equation*}
    By Theorem \ref{beta invariant}, $S$ is not $K$-semistable.
\end{proof}

\begin{lemma}
    Let $S\subset\PP(1, 1, n+1, m+1)$ be a quasi-smooth weighted hypersurface of degree $n + m + 2$ where $n$ and $m$ are non-negative integers such that $n < m$. Then $S$ is not $K$-semistable.
\end{lemma}

\begin{proof}
    By a suitable coordinate change, we can assume that $S$ is give by a quasi-homogeneous polynomial
    \begin{equation*}
        tz + f(x,y) = 0
    \end{equation*}
    in $\PP(1, 1, n+1, m+1)$, where $f(x,y)=\prod_{i=1}^{n+m+2}(a_{i}x+b_iy)$ is a homogeneous polynomial that is smooth in $\PP^1$, with weights $\wt(x)=\wt(y)=1$, $\wt(z)=n+1$ and $\wt(t)=m+1$. Then $S$ is singular at the points $\msp_{z}$ of type $\frac{1}{n+1}(1,1)$ and $\msp_t$ of type $\frac{1}{m+1}(1,1)$.

    In a neighborhood of $\msp_t$, we may regard that $x$ and $y$ are local weighted coordinates. Let $\pi\colon Y\to S$ be the weighted blow-up at $\msp_t$ with weights $\wt(x) = 1$ and $\wt(y) = 1$. Then we have
    \[
        K_Y \equiv \pi^*(K_S) - \frac{m - 1}{m + 1}E
    \]
    where $E$ is the exceptional divisor of $\pi$. Let $H_z$ be the hyperplane that is cut out by $z=0$ in $S$. We write
    \[
        H_z = \sum_{i=1}^{n+m+2}L_{i}
    \]
    where $L_i$ is the curve that is given by $z=a_{i}x+b_{i}y = 0$ in $\PP(1,1,n+1,m+1)$. Since
    \begin{equation*}
        \pi^*(L_i) = \bar{L}_i + \frac{1}{m+1}E
    \end{equation*}
    where $\bar{L}_i$ be the strict transform of $L_i$, we have
    \begin{equation*}
        \pi^*(H_z) = \sum_{i=1}^{n+m+2}\bar{L}_{i} + \frac{n+m+2}{m+1}E.
    \end{equation*}
    Then we have the following intersection numbers:
    \[
        \bar{L}_i^2 = -1, \qquad \bar{L}_{i}\cdot \bar{L}_{j} = 0,\qquad \bar{L}_i\cdot E = 1
    \]
    where $i\neq j$.

    We claim that 
    \begin{equation*}
        \beta(E) = A_S(E)(-K_S)^2 - \int_{0}^{\tau (E)}\vol(\pi^*(-K_S) - \lambda E) \dd\lambda
    \end{equation*}
    is not positive. To prove this, we consider the following $\QQ$-divisor
    \[
        \pi^*\left(\frac{2}{1+n}H_z\right)-\lambda E = \frac{2}{1+n}\sum_{i=1}^{n+m+2}\bar{L}_{i} + \left(\frac{2(n+m+2)}{(n+1)(m+1)} - \lambda\right)E
    \]
    where $\lambda$ is non-negative number. Since $\bar{L}_i^2 < 0$ for each $\bar{L}_i$, we have $\tau(E) = \frac{2(n+m+2)}{(n+1)(m+1)}$, that is, $\vol(\pi^*(-K_S) - \lambda E) = 0$ for $\tau(E) < \lambda$.

    We have
    \[
        \left(\frac{2}{n+1}\sum_{i=1}^{n+m+2}\bar{L}_{i} + \left(\frac{2(n+m+2)}{(n+1)(m+1)} - \lambda\right)E\right)\cdot \bar{L}_i = \frac{2}{m+1} - \lambda
    \]
    for each $\bar{L}_i$. They imply that $\pi^*(-K_S)-\lambda E$ is nef for $\lambda\leq\frac{2}{m+1}$. Thus we obtain the following.
    \[
        \vol(\pi^*(-K_S) - \lambda E)=(\pi^*(-K_S) - \lambda E)^2=\frac{4(n+m+2)}{(n+1)(m+1)} - (m+1)\lambda^2
    \]
    for $\lambda\leq\frac{2}{m+1}$.
    
    We assume that $\frac{2}{m+1}\leq\lambda\leq \tau(E)$. Since $(\bar{H}_z + E)\cdot \bar{L}_i = 0$ for each $\bar{L}_i$ and the intersection form of $\bar{L}_1,\ldots ,\bar{L}_{n+m+2}$ is negative definite, the Zariski decomposition of $\pi^*(-K_S) - \lambda E$ is
    \[
        \left(\frac{2(n+m+2)}{(n+1)(m+1)} - \lambda\right)(\bar{H}_z + E) + \left(\lambda - \frac{2}{m+1}\right)\bar{H}_z.
    \]
    From this we have
\begin{equation*}
    \vol(\pi^*(-K_S) - \lambda E)= \vol\left(\left(\frac{2(n+m+2)}{(n+1)(m+1)} - \lambda\right)(\bar{H}_z + E)\right)= \left(\frac{2(n+m+2)}{(n+1)(m+1)} - \lambda\right)^2(n+1).
\end{equation*}
    
    Then we have
    \begin{IEEEeqnarray*}{rCl}
        \IEEEeqnarraymulticol{3}{l}
        {
            \displaystyle\int_{0}^{\tau(E)}\vol(\pi^*(-K_S) - \lambda E) \dd\lambda
        }\\
            & = & \displaystyle\int_{0}^{\frac{2}{m+1}}\frac{4(n+m+2)}{(n+1)(m+1)} - (m+1)\lambda^2~\mathrm{d}\lambda + \displaystyle\int_{\frac{2}{m+1}}^{\tau(E)}\left(\frac{2(n+m+2)}{(n+1)(m+1)} - \lambda\right)^2(n+1)~\mathrm{d}\lambda\\ 
            & = & \frac{8(n+m+2)}{(n+1)(m+1)^2} - \frac{8}{3(m+1)^2} + \frac{8}{3(n+1)^2}.
    \end{IEEEeqnarray*}    

    Moreover we have $A_S(E)(-K_S)^2 = \frac{8(n+m+2)}{(n+1)(m+1)^2}$. As a result, we see that
    \begin{equation*}
        \beta(E) = \frac{8}{3(m+1)^2} - \frac{8}{3(n+1)^2} < 0
    \end{equation*}
    for $0\leq n < m$. By Theorem \ref{beta invariant}, $S$ is not $K$-semistable.
\end{proof}

\begin{lemma}
    Let $S\subset \PP(1,3,3n + 4, 3n + 5)$ be a quasi-smooth weighted hypersurface of degree $6n + 11$ where $n$ is non-negative integer. Then $S$ is not $K$-semistable.
\end{lemma}

\begin{proof}
By a suitable coordinates change, we can assume that $S$ is given by a quasi-homogeneous polynomial
\begin{equation*}
    t^2x + ty^{n+2} + z^2y + xf_{6n+10}(x,y,z,t) = 0
\end{equation*}
where $f_{6n+10}(x,y,z,t)$ is a quasi-homogeneous polynomial of degree $6n+10$ with weights $\wt(x) = 1$, $\wt(y) = 3$ $\wt(z) = 3n+4$ and $\wt(t) = 3n+5$. 

Let $H_x$ be the hyperplane that is cut out by $x=0$ in $S$. Then it is isomorphic to the variety given by
\begin{equation*}
    (ty^{n+1} + z^2)y = 0
\end{equation*}
in $\PP(3,3n+4,3n+5)$. From this we can see that the following holds:
\begin{equation*}
    H_x = L + R
\end{equation*}
where $L$ and $R$ are the curves given by $x = y = 0$ and $x = ty^{n+1} + z^2 = 0$ in $\PP(1,3,3n + 4, 3n + 5)$, respectively. We have the following intersection numbers:
\[
    L\cdot R = \frac{2}{3n+5},\qquad H_x\cdot L = \frac{1}{(3n+4)(3n+5)},\qquad H_x\cdot R = \frac{2}{3(3n+5)},
\]

\[
    L^2 = -\frac{6n+7}{(3n+4)(3n+5)},\qquad R^2 = -\frac{4}{3(3n+5)}.
\]

Consider the singular point $\msp_t$ of type $\frac{1}{3n+5}(2, n+1)$ of $S$. In the chart defined by $t=1$, $S$ is given by
\begin{equation*}
    x + y^{n+2} + z^2y + xf_{6n+10}(x,y,z,1) = 0.
\end{equation*}
From above equation, we regard $y$ and $z$ as local weighted coordinates in a neighborhood of $\msp_t$ with weights $\wt(y) = 2$ and $\wt(z) = n+1$. 

Let $\pi\colon Y \to S$ be the weighted blow-up at $\msp_t$ with weights $\wt(y) = 1$ and $\wt(z) = n+1$. Then we have
\begin{equation*}
    K_Y \equiv \pi^*(K_S) - \frac{2n+2}{3n+5}E
\end{equation*}
where $E$ is the exceptional divisor of $\pi$. Since
\begin{equation*}
    \pi^*(L) = \bar{L} + \frac{2}{3n+5}E,\qquad \pi^*(R) = \bar{R} + \frac{2n+2}{3n+5}E,\qquad E^2 = -\frac{3n+5}{2n+2}
\end{equation*}
where $\bar{L}$ and $\bar{R}$ are the strict transforms of $L$ and $R$, respectively, we have the following intersection numbers:
\begin{equation*}
    \bar{L}^2 = -\frac{2n + 3}{(n+1)(3n+4)},\qquad \bar{R}^2 = -\frac{2}{3},\qquad \bar{L}\cdot\bar{R} = 0,\qquad  \bar{L}\cdot E = \frac{1}{n+1},\qquad \bar{R}\cdot E = 1.
\end{equation*}

To compute $\beta(E)$ we must compute $\vol(\pi^*(-K_S) - \lambda E)$ for any non-negative number $\lambda$. So we consider the following.
\begin{equation*}
    \pi^*(-K_S)-\lambda E \equiv 2\bar{L} + 2\bar{R} + \left(\frac{4n+8}{3n+5} - \lambda \right)E.
\end{equation*}
Since $\bar{L}^2 < 0$ and $\bar{R} < 0$ we have $\vol(\pi^*(-K_S) - \lambda E) = 0$ for $\frac{4n+8}{3n+5} < \lambda$. Meanwhile, we have
\begin{equation*}
    \left(2\bar{L} + 2\bar{R} + \left(\frac{4n+8}{3n+5} - \lambda\right)E\right)\cdot \bar{L} = \frac{2}{(3n+4)(3n+5)} - \frac{\lambda}{n+1}\geq 0
\end{equation*}
and
\begin{equation*}
    \left(2\bar{L} + 2\bar{R} + \left(\frac{4n+8}{3n+5} - \lambda\right)E\right)\cdot \bar{R} = \frac{4}{3(3n+5)} - \lambda\geq 0
\end{equation*}
for $\lambda\leq \frac{2n+2}{(3n+4)(3n+5)}$. Thus $\pi^*(-K_S) - \lambda E$ is nef and
\begin{equation*}
    \vol(\pi^*(-K_S) - \lambda E) = (\pi^*(-K_S) - \lambda E)^2 = \frac{4(6n+11)}{3(3n+4)(3n+5)} - \frac{3n+5}{2n+2}\lambda^2
\end{equation*}
for $\lambda\leq \frac{2n+2}{(3n+4)(3n+5)}$. Next, we assume that $\frac{2n+2}{(3n+4)(3n+5)}\leq \lambda \leq \frac{4}{3(3n+5)}$. We set
\begin{equation*}
    P = \frac{3n+4}{2n+3}\left(\frac{4n+8}{3n+5} - \lambda\right)\bar{L} + 2\bar{R} + \left(\frac{4n+8}{3n+5} - \lambda\right)E.
\end{equation*}
Since $P\cdot \bar{L} = 0$, $P\cdot \bar{R} > 0$ and $P\cdot E > 0$, $P$
is nef. Moreover we see that the Zariski decomposition of $\pi^*(-K_S) - \lambda E$ is
\begin{equation*}
    P + \left(2-\frac{3n+4}{2n+3}\left(\frac{4n+8}{3n+5} - \lambda\right)\right)\bar{L}
\end{equation*}
where $P$ is the positive part. Thus we have
\begin{equation*}
    \vol(\pi^*(-K_S) - \lambda E) = P^2 = -\frac{8}{3} + 4\left(\frac{4n+8}{3n+5} - \lambda\right) - \frac{6n+7}{4n+6}\left(\frac{4n+8}{3n+5} - \lambda\right)^2
\end{equation*}
for $\frac{2n+2}{(3n+4)(3n+5)}\leq \lambda \leq \frac{4}{3(3n+5)}$.

Finally we consider case that $\frac{4}{3(3n+5)}\leq \lambda \leq \frac{4n+8}{3n+5}$. We set
\begin{equation*}
    P_2 = \frac{3n+4}{2n+3}\left(\frac{4n+8}{3n+5} - \lambda\right)\bar{L} + \frac{3}{2}\left(\frac{4n+8}{3n+5} - \lambda\right)\bar{R} + \left(\frac{4n+8}{3n+5} - \lambda\right)E.
\end{equation*}
Since $P_2\cdot \bar{L} = 0$, $P_2\cdot \bar{R}=0$ and $P_2\cdot E > 0$, $P_2$ is nef. Moreover the intersection form of $\bar{L}$ and $\bar{R}$ is negative definite. Thus
\begin{equation*}
    \pi^*(-K_S) - \lambda E \equiv P_2 + \left(2-\frac{3n+4}{2n+3}\left(\frac{4n+8}{3n+5} - \lambda\right)\right)\bar{L} + \left(2 - \frac{3}{2}\left(\frac{4n+8}{3n+5} - \lambda\right)\right)\bar{R}
\end{equation*}
is the Zariski decomposition of $\pi^*(-K_S) - \lambda E$ where $P_2$ is the positive part. It implies that
\begin{equation*}
    \vol(\pi^*(-K_S) - \lambda E) = P_2^2 = \left(\frac{3}{2} - \frac{6n+7}{4n+6}\right)\left(\frac{4n+8}{3n+5} - \lambda \right)^2
\end{equation*}
for $\frac{4}{3(3n+5)}\leq \lambda \leq \frac{4n+8}{3n+5}$. Therefore we have
\begin{equation*}
    \int_0^{\infty}\vol(\pi^*(-K_S) - \lambda E) d\lambda = \frac{8}{27}\left(\frac{108n^3 + 594n^2 + 1053n + 601}{(3n+4)^2(3n+5)^2}\right).
\end{equation*}
As a result, we have
\begin{equation*}
    \begin{array}{lll}
        \beta(E) & = & A_S(E)(-K_S)^2 - \int_0^{\infty}\vol(\pi^*(-K_S) - \lambda E) d\lambda\\ \\
        & = & - \dfrac{1}{27}\left(\dfrac{702n^3 + 3753n^2 + 6489n + 3620}{(3n+4)^2(3n+5)^2} \right) < 0.
    \end{array}
\end{equation*}
By Theorem \ref{beta invariant}, $S$ is not $K$-semistable.
\end{proof}


\begin{thebibliography}{99}
\bibitem{A02}
C.~Araujo, 
\emph{K\"{a}hler-Einstein metrics for some quasi-smooth log del Pezzo surfaces},
Trans. Amer. Math. Soc. 354 (2002) 4303–3312.

\bibitem{BG08}
C. Boyer, K. Galicki, \emph{Sasakian geometry}, Oxford University Press, 2008.

\bibitem{BGN03}
 C. Boyer, K. Galicki, M. Nakamaye, \emph{On the geometry of Sasakian–Einstein 5-manifolds}, Math. Annalen
325 (2003), 485–524.





\bibitem{BJ20}
H.~Blum, M.~Jonsson, 
\emph{Thresholds, valuations, and $K$-stability}, 
Adv. Math. 365 (2020). 

\bibitem{BauerKuronyaSzemberg}
Th. Bauer, A. K\"uronya, T. Szemberg, \emph{Zariski chambers, volumes, and stable base loci}, J. Reine Angew. Math. \textbf{576} (2004), 209--233.

\bibitem{CPS10}
I.~Cheltsov, J.~Park, C.~Shramov, 
\emph{Exceptional del Pezzo hypersurfaces}, 
J. Geom. Anal. 20(4) (2010) 787–816.

\bibitem{CPS18}
I.~Cheltsov, J.~Park, C.~Shramov,
\emph{Delta invariants of singular del Pezzo surfaces},
J. Geom. Anal. to appear




\bibitem{F19}
K.~Fujita, 
\emph{On the uniform $K$-stability for some asymptotically log del Pezzo surfaces}, preprint, arXiv:1907.04998 (2019).

\bibitem{F19b} 
K.~Fujita,
\emph{A valuative criterion for uniform K-stability of $\QQ$-Fano varieties},
J. Reine Angew. Math 2019.751 (2019), 309--338.




\bibitem{JK01}
J.~Johnson, J.~ Koll\'ar, 
\emph{K\"ahler-Einstein metrics on log del Pezzo surfaces in weighted projective 3-spaces}, 
Ann. de l’Institut Fourier 51 (2001), 69--79.




\bibitem{Laz-positivity-in-AG}
R. Lazarsfeld, \emph{Positivity in Algebraic Geometry, I, II}, Springer 2004.

\bibitem{LTW17}
C.~Li, G.~Tian, F.~Wang, 
\emph{On Yau--Tian--Donaldson conjecture for singular Fano varieties}, 
preprint, arXiv:1711.09530 (2017).




\bibitem{Prokhorov}
Yu.~Prokhorov, \emph{On the Zariski Decomposition Problem}, Proc. Steklov Inst. Math. \textbf{240} (2003), 37--65.

\bibitem{PW19}
J.~Park, J.~Won, 
\emph{Simply connected Sasaki-Einstein rational homology 5-spheres}, 
arXiv:1905.13304 (2019).

\end{thebibliography}
\end{document}